\documentclass[11pt]{amsart}
\usepackage{amsmath,amssymb,latexsym,soul,cite}
\usepackage{xcolor,enumitem,graphicx}

\definecolor{CherryRed}{rgb}{0.7,0.1,0.1}
\definecolor{EgyptBlue}{rgb}{0.1,0.1,0.6}

\usepackage[colorlinks=true,urlcolor=olive,
citecolor=CherryRed,linkcolor=EgyptBlue,linktocpage,pdfpagelabels,
bookmarksnumbered,bookmarksopen]{hyperref}
\usepackage[english]{babel}
\usepackage[T1]{fontenc}
\usepackage[hyperpageref]{backref}

\usepackage[left=2.8cm,right=2.8cm,top=2.9cm,bottom=2.9cm]{geometry}

\numberwithin{equation}{section}

\newtheorem{thm}{Theorem}[section]

  \theoremstyle{plain}
  \newtheorem{lem}[thm]{Lemma}
  \theoremstyle{plain}
  \newtheorem{prop}[thm]{Proposition}
  \theoremstyle{plain}
  \newtheorem{cor}[thm]{Corollary}
  \theoremstyle{plain}
  
    \theoremstyle{definition}
\newtheorem{rem}[thm]{Remark}

\newcommand{\R}{{\mathbb R}}
\newcommand{\N}{{\mathbb N}}
\newcommand{\Z}{{\mathbb Z}}

\title{On multiplicity of eigenvalues and symmetry of eigenfunctions of the $p$-Laplacian}

\author[B.~Audoux]{Benjamin Audoux}
\author[V.~Bobkov]{Vladimir Bobkov}
\author[E.~Parini]{Enea Parini}

\address[B.~Audoux, E.~Parini]{Aix Marseille Univ, CNRS, Centrale Marseille, I2M, 39 Rue Frederic Joliot Curie, 13453 Marseille, France}
\email{benjamin.audoux@univ-amu.fr} \email{enea.parini@univ-amu.fr}
\address[V.~Bobkov]{University of West Bohemia, Faculty of Applied Sciences, Department of Mathematics and NTIS, Univerzitn\'i 8, 306 14 Plze\v{n}, Czech Republic}
\email{bobkov@kma.zcu.cz}

\thanks{}
\subjclass[2010]{
35J92;  	
35P30;  	
35A15;  	
35A16;  	
55M25;  	
35B06;  	
}
\keywords{$p$-Laplacian; nonlinear eigenvalues; Krasnoselskii genus; symmetry; multiplicity; degree of map}

\begin{document}

\begin{abstract}
	We investigate multiplicity and symmetry properties of higher eigenvalues and eigenfunctions of the $p$-Laplacian under homogeneous Dirichlet boundary conditions on certain symmetric domains $\Omega \subset \R^N$. By means of topological arguments, we show how symmetries of $\Omega$ help to construct subsets of $W_0^{1,p}(\Omega)$ with suitably high Krasnosel'ski\u{\i} genus. In particular, if $\Omega$ is a ball $B \subset \mathbb{R}^N$, we obtain the following chain of inequalities:
	$$
	\lambda_2(p;B) \leq \dots \leq \lambda_{N+1}(p;B) \leq \lambda_\ominus(p;B).
	$$
	Here $\lambda_i(p;B)$ are variational eigenvalues of the $p$-Laplacian on $B$, and 	
	$\lambda_\ominus(p;B)$ is the eigenvalue which has an associated eigenfunction whose nodal set is an equatorial section of $B$. If $\lambda_2(p;B)=\lambda_\ominus(p;B)$, as it holds true for $p=2$, the result implies that the multiplicity of the second eigenvalue is at least $N$. In the case $N=2$, we can deduce that any third eigenfunction of the $p$-Laplacian on a disc is nonradial. 
	The case of other symmetric domains and the limit cases $p=1$, $p=\infty$ are also considered.
\end{abstract}

\maketitle

\section{Introduction and main results}\label{sec:inro}

Let $\Omega \subset \R^N$, $N \geq 2$, be a bounded, open domain, and let $p>1$. We say that $u \in W^{1,p}_0(\Omega) \setminus \{0\}$ is an eigenfunction of the $p$-Laplacian associated to the eigenvalue $\lambda \in \mathbb{R}$ if it is a weak solution of
\begin{equation}\label{eq:D}
 \left\{
 \begin{array}{r c l l} -\Delta_p u & = & \lambda |u|^{p-2}u & \text{in }\Omega, \\ u & = & 0 & \text{on }\partial \Omega,
 \end{array}
 \right.
\end{equation}
where $\Delta_p u = \text{div}(|\nabla u|^{p-2}\nabla u)$. If $p=2$, \eqref{eq:D} is the well-known eigenvalue problem for the Laplace operator. The first eigenvalue $\lambda_1(p;\Omega)$ of the $p$-Laplacian is defined as
\begin{equation}\label{firsteigenvalue}
\lambda_1(p;\Omega) = \min_{u \in \mathcal{S}_p} \int_\Omega |\nabla u|^p \, dx,
\end{equation}
where
$$ 
\mathcal{S}_p := \{u \in W^{1,p}_0(\Omega)\,\big|\, \|u\|_{L^p(\Omega)} = 1\}.
$$
Besides the first eigenvalue, in the linear case $p=2$, the standard Courant-Fisher minimax formula
\begin{equation}\label{eq:eigenvalue_linear}
\lambda_{k}(2;\Omega) = \min_{X_k} \max_{u \in X_k \cap \mathcal{S}_2} \int_\Omega |\nabla u|^2 \, dx, 
\quad k \in \mathbb{N},
\end{equation}
provides a sequence of eigenvalues which exhausts the spectrum of the Laplacian, cf.\ \cite[Theorem 8.4.2]{attouchbuttazzo}. 
In \eqref{eq:eigenvalue_linear}, the minimum is taken over subspaces $X_k \subset W_0^{1,2}(\Omega)$ of dimension $k$. 
However, for $p\neq 2$ the problem is nonlinear, and it is necessary to make use of a different method. A sequence of \textit{variational eigenvalues} can be obtained by means of the following minimax variational principle. 
Let $\mathcal{A} \subset W^{1,p}_0(\Omega)$ be a \textit{symmetric set}, i.e., if $u \in \mathcal{A}$, then $-u \in \mathcal{A}$. Define the {\it Krasnosel'ski\u{\i} genus} of $\mathcal{A}$ as
$$ 
\gamma(\mathcal{A}):=\inf\{k\in\mathbb{N} \,\big|\, \exists \mbox{ a continuous odd map } f:\mathcal{A} \to S^{k-1}\}
$$
with the convention $\gamma(\mathcal{A}) = +\infty$ if, for every $k \in \mathbb{N}$, no continuous odd map $f:\mathcal{A} \to S^{k-1}$ exists.
Here $S^{k-1}$ is a $(k-1)$-dimensional sphere.
For $k \in \mathbb{N}$ we define
$$ 
\Gamma_k(p) := \left\{\mathcal{A} \subset \mathcal{S}_p \,\big|\,  \mathcal{A} \mbox{ symmetric and compact},\, \gamma(\mathcal{A})\geq k\right\}
$$
and
\begin{equation}\label{highereigenvalues} 
\lambda_k(p;\Omega) := \inf_{\mathcal{A} \in \Gamma_k(p)} \max_{u \in \mathcal{A}} \int_\Omega |\nabla u|^p \, dx. 
\end{equation}
It is known that each $\lambda_k(p;\Omega)$ is an eigenvalue and
\begin{equation*}
0 < \lambda_1(p;\Omega) < \lambda_2(p;\Omega) \leq \dots \leq \lambda_k(p;\Omega) \to +\infty
\quad \text{as } k \to +\infty,
\end{equation*}
see \cite[\S 5]{garciaazoreroperal}. 
However, it is not known if the sequence $\{\lambda_k(p;\Omega)\}_{k=1}^{+\infty}$ exhausts all possible eigenvalues, except for the case $p=2$, where the eigenvalues in \eqref{highereigenvalues} coincide with the eigenvalues in \eqref{eq:eigenvalue_linear}, see, e.g., \cite[Proposition 4.7]{cuesta} or \cite[Appendix~A]{brascoparinisquassina}. 
It has to be observed that the definitions of $\lambda_1(p;\Omega)$ by \eqref{firsteigenvalue} and \eqref{highereigenvalues} are consistent. 
The associated first eigenfunction is unique modulo scaling and has a strict sign in $\Omega$ (cf.~\cite{bellonikawohl,vazquez}), while eigenfunctions associated to any other eigenvalue must necessarily be sign-changing (see, e.g., \cite[Lemma~2.1]{kawohllindqvist}). 
Therefore, it makes sense to define the \textit{nodal domains} of an eigenfunction $u$ as the connected components of the set $\{x \in \Omega: u(x) \neq 0\}$, and the \textit{nodal set} of $u$ as $\{x \in \Omega: u(x) = 0\}$.
The version of the Courant nodal domain theorem for the $p$-Laplacian obtained in \cite{drabekrobinson} states that any eigenfunction associated to $\lambda_k(p;\Omega)$ with $k \geq 2$ has at most $2k-2$ nodal domains.
In particular, any eigenfunction associated to $\lambda_2(p;\Omega)$ has exactly two nodal domains.
Moreover, since there are no eigenvalues between $\lambda_1(p;\Omega)$ and $\lambda_2(p;\Omega)$ \cite{anane}, the latter is indeed the second eigenvalue. 

\medskip

For the sake of simplicity, in the following we will restrict our attention mainly to the case where $\Omega = B^N$ is an open $N$-ball centred at the origin. In the linear case $p=2$, the eigenfunctions of the Laplace operator on $B^N$ are given explicitly by means of Bessel functions and spherical harmonics, and therefore it can be seen that the first eigenfunction is radially symmetric, while the nodal set of any second eigenfunction is an equatorial section of the ball; moreover, the following multiplicity result holds true:
\begin{equation}\label{eq:chain_for_linear_case}
\lambda_1(2;B^N) < \lambda_2(2;B^N) = \dots = \lambda_{N+1}(2;B^N) < \lambda_{N+2}(2;B^N),
\end{equation}
see, for instance, the discussion in \cite{helffer}.
In contrast, in the nonlinear case $p \neq 2$ much less is known. While it is relatively easy to show that the first eigenfunction is still radially symmetric by means of Schwarz symmetrization, symmetry properties of second eigenfunctions, as well as the multiplicity of the second eigenvalue, are not yet completely understood. For instance, it is known only that second eigenfunctions can not be radially symmetric; this was shown in the planar case in \cite{parini} for $p$ close to $1$, and later in \cite{benediktdrabekgirg} for general $p > 1$. The result was finally generalized to any dimension in \cite{anoopdrabeksasi}. The notion of multiplicity itself needs to be clarified in the nonlinear case. We say that the variational eigenvalue $\lambda_{k}(p;\Omega)$ has multiplicity $m$ if there exist $m$ variational eigenvalues $\lambda_{l}, \dots, \lambda_{l+m-1}$ with $l \leq k \leq l+m-1$ such that
\begin{equation}\label{eq:multiplicity}
\lambda_{l-1}(p;\Omega) < \lambda_{l}(p;\Omega) = \dots = \lambda_{k}(p;\Omega) = \dots = 
\lambda_{l+m-1}(p;\Omega) < \lambda_{l+m}(p;\Omega).
\end{equation}
We point out that we are not aware of any multiplicity results for higher eigenvalues of the $p$-Laplacian. 

\medskip

Despite the deficit of information about symmetry properties of variational eigenfunctions, it is possible to consider eigenvalues (possibly non-variational) with associated eigenfunctions which respect certain symmetries of $B^N$. For instance, the existence of a sequence of eigenvalues
$$ 
0 < \mu_1(p;B^N) < \mu_2(p;B^N) < \dots < \mu_k(p;B^N) \to +\infty
\quad \text{as } k \to +\infty,
$$
corresponding to \emph{radial eigenfunctions} has been shown, for instance, in \cite{delpinomanasevich}. Each radial eigenfunction associated to $\mu_k(p;B^N)$ is unique modulo scaling and possesses exactly $k$ nodal domains. The latter implies that $\lambda_k(p;B^N) \leq \mu_k(p;B^N)$ for any $k \in \mathbb{N}$ and $p>1$ (see Lemma~\ref{lem:Krasnoselskii_by_scaling} below). The above-mentioned results about radial properties of first and second eigenfunctions, together with \cite[Theorem~1.1]{bobkovdrabek}, can therefore be stated as 
$$ 
\lambda_1(p;B^N) = \mu_1(p;B^N) 
\quad 
\text{and} 
\quad \lambda_k(p;B^N)<\mu_k(p;B^N)
\quad \text{for all } p>1 \text{ and } k \geq 2.
$$
Another sequence of eigenvalues
$$ 
0 < \tau_1(p;B^N) < \tau_2(p;B^N) < \dots < \tau_k(p;B^N) \to +\infty
\quad \text{as } k \to +\infty,
$$
was considered in \cite[Theorem~1.2]{anoopdrabeksasi}. Here $\tau_k(p;B^N)$ is constructed in such a way that it has an associated \textit{symmetric eigenfunction}\footnote{\label{footref1}
We use the adjective ``symmetric'' to distinguish this eigenfunction from the radial one, since $\mu_k(p;B^N)$ and $\tau_k(p;B^N)$ can be equal to each other and hence might have associated eigenfunctions with not appropriate nodal structures, see \cite[Corollary~1.3 and Theorem~1.4]{bobkovdrabek}.} whose nodal domains are spherical wedges of angle $\frac{\pi}{k}$; see also Section~\ref{sec:Eigenvalue_auxiliary_facts} below, where a generalization of this sequence to other symmetric domains is given. 
In particular, the nodal set of any symmetric eigenfunction associated to $\tau_1(p;B^N)$ is an equatorial section of $B^N$.
By construction, a symmetric eigenfunction associated to $\tau_k(p;B^N)$ has $2k$ nodal domains, which implies that
$$
\lambda_{2k}(p;B^N) \leq \tau_k(p;B^N)
\quad 
\text{for any } k \in \mathbb{N} \text{ and } p > 1.
$$
At the same time, in the linear case, one can easily use the Courant-Fisher variational principle \eqref{eq:eigenvalue_linear} to show (see Remark~\ref{rem:2k+1_linear} below) that at least
\begin{equation}\label{eq:l2k<tk}
\lambda_{2k}(2;B^N) \leq \lambda_{2k+1}(2;B^N) \leq \tau_k(2;B^N)
\quad 
\text{for any } k \in \mathbb{N}.
\end{equation} 

The generalization of even such simple facts as \eqref{eq:chain_for_linear_case} and \eqref{eq:l2k<tk} to the nonlinear case $p \neq 2$ meets certain difficulties. 
The main obstruction consists in the following fairly common problem: 
\medskip
\begin{center}
	\textit{How to obtain a symmetric compact set $A \subset \mathcal{S}_p$ with suitably high Krasnosel'ski\u{\i} genus, and, at the same time, with suitably low value $\max\limits_{u \in A} \int_\Omega |\nabla u|^p \, dx$?}
\end{center}
\medskip
In the linear case, the consideration of subspaces spanned by the first $k$ eigenfunctions $\varphi_1, \dots, \varphi_k$ directly solves this problem. 
Let us sketchily describe the approach supposing that we want to prove the multiplicity in \eqref{eq:chain_for_linear_case} using the definition \eqref{highereigenvalues} only. 
Let $\varphi_1$ and $\varphi_2$ be a first and a second eigenfunction of the Laplacian on $B^N$, respectively, such that $\|\varphi_i\|_{L^2(B^N)} = 1$ for $i=1,2$. Since $B^N$ and the Laplace operator are rotation invariant, we see that $\varphi_2$ generates $N$ linearly independent second eigenfunctions $\varphi_2, \dots, \varphi_{N+1}$ whose nodal sets are equatorial sections of $B^N$ orthogonal to each other.
Consider the set
\begin{equation}\label{eq:A_k}
\mathcal{B}_2 := 
\bigg\{
\sum_{i=1}^{N+1} \alpha_i \varphi_i\,\big|\, \sum_{i=1}^{N+1} |\alpha_i|^2 = 1
\biggr\}.
\end{equation}
Evidently, $\mathcal{B}_2$ is symmetric and compact, and it is not hard to show that $\gamma(\mathcal{B}_2) = N+1$. Moreover, since all $\varphi_1, \dots, \varphi_{N+1}$ are mutually orthogonal with respect to $L^2$-inner product, we get $\mathcal{B}_2 \subset \mathcal{S}_2$. Indeed, 
\begin{equation}\label{eq:orthogonaliry}
\|u\|_{L^2(B^N)}^2 = \sum_{i=1}^{N+1} \alpha_i^2 \, \|\varphi_i\|_{L^2(B^N)}^2 = 1
\quad \text{for any } u \in \mathcal{B}_2.
\end{equation}
Therefore, $\mathcal{B}_2 \in \Gamma_{N+1}(2)$, and, using again the orthogonality, we obtain
\begin{equation*}
\lambda_{N+1}(2;B^N) \leq \max_{u \in \mathcal{B}_2} \int_{B^N} |\nabla u|^2 \, dx \leq 
\max_{\alpha_1^2 + \dots + \alpha_{N+1}^2 = 1}\sum_{i=1}^{N+1} \alpha_i^2 \, \lambda_2(2;B^N) \|\varphi_i\|_{L^2(B^N)}^2 = \lambda_2(2;B^N),
\end{equation*}
which leads to the desired chain of equalities in \eqref{eq:chain_for_linear_case}. 

However, this approach does not work well enough in the nonlinear case $p \neq 2$. 
First of all, we do not know if a second eigenfunction has an equatorial section of $B^N$ as its nodal set. This can be overcome by considering a symmetric eigenfunction $\Psi_1$ associated to $\tau_1(p;B^N)$. Using the first eigenfunction $\varphi_1$, symmetric eigenfunction $\Psi_1$, and noting that the $p$-Laplacian is rotation invariant for $p>1$, we can produce $(N+1)$ linearly independent eigenfunctions as above and define a symmetric compact set $\mathcal{B}_p$ analogously to \eqref{eq:A_k}. 
Moreover, similarly to \cite[Lemma~2.1]{huang} it can be shown that $\gamma(\mathcal{B}_p) = N+1$. However, the lack of the $L^2$-orthogonality prevents to achieve $\mathcal{B}_p \subset \mathcal{S}_p$ as in \eqref{eq:orthogonaliry}, and further normalization of $\mathcal{B}_p$ increases the value $\max\limits_{u \in \mathcal{B}_p} \int_{B^N} |\nabla u|^p \, dx$.\footnote{A similar approach was used in \cite[Section~2]{huang}. However, this approach also does not give a necessarily small upper bound for $\max\limits_{u \in A_{k}(p)} \int_\Omega |\nabla u|^p \, dx$ due to a gap in the proof of \cite[Lemma~2.3]{huang}. Namely, it is assumed that $\|u\|_{L^p(\Omega)}=1$ for any $u \in A_k(p)$ which might not be correct.}

\medskip

Another usual approach to obtain sets of higher Krasnosel'ski\u{\i} genus for general $p>1$ is based on the independent \textit{scaling} of nodal components of a function, cf.\ Lemma~\ref{lem:Krasnoselskii_by_scaling} below. Assume that some $w \in W_0^{1,p}(\Omega)$ can be represented as $w = w_1 +\dots+ w_k$, where all $w_i \in \mathcal{S}_p$ and they are disjointly supported. 
Considering the set
$$
\mathcal{C}_k = 
\biggl\{
\sum_{i=1}^{k} \alpha_i w_i\,\big|\, \sum_{i=1}^{k} |\alpha_i|^p = 1
\biggr\},
$$
we easily achieve that $\mathcal{C}_k \in \Gamma_k(p)$. 
However, as before, the disadvantage of this approach is that $\max\limits_{u \in \mathcal{C}_k} \int_{\Omega} |\nabla u|^p \, dx$ cannot be made, in general, appropriately small. 

\medskip

In this article, we present a variation of the above-mentioned approaches. Namely, using the symmetries of $\Omega$, we combine the scaling of nodal components of an eigenfunction with its rotations, which allows us to find a set $\mathcal{A} \in \Gamma_k(p)$ for appropriately big $k \in \mathbb{N}$, while keeping control of the value $\max\limits_{u \in \mathcal{A}} \int_{\Omega} |\nabla u|^p \, dx$. 
By virtue of this fact, we obtain the following generalizations of \eqref{eq:chain_for_linear_case} and \eqref{eq:l2k<tk}, which can be seen as a step towards exact multiplicity results for nonlinear variational higher eigenvalues.

\medskip

\begin{thm}\label{thm:main1}
	Let $\Omega \subset \R^N$ be a radially symmetric bounded domain, $N \geq 2$. Let $p>1$, $k \geq 1$ and let $\tau_k(p;\Omega)$ be defined as in \eqref{definitiontau}. Then the following inequalities are satisfied:
	\begin{align}
	\label{eq:ln+1<t2}
	\lambda_2(p;\Omega) \leq \dots \leq \lambda_{N+1}(p;\Omega) &\leq \tau_1(p;\Omega);\\
	\label{eq:l2k+1tk1}
	\lambda_{2k}(p;\Omega) \leq \lambda_{2k + 1}(p;\Omega) &\leq \tau_k(p;\Omega).
	\end{align}
\end{thm}

\medskip

Theorem \ref{thm:main1} implies that, if $\lambda_2(p;\Omega)=\tau_1(p;\Omega)$, then the second eigenvalue has multiplicity at least $N$. It is also meaningful to emphasize that the inequalities \eqref{eq:ln+1<t2} \textit{do not} imply that eigenfunctions associated to $\lambda_3(p;B^N), \dots, \lambda_{N+1}(p;B^N)$ are nonradial. Indeed, to the best of our knowledge, the inequality $\tau_1(p;B^N) < \mu_2(p;B^N)$ is not proved yet for general $p > 1$ and $N \geq 3$. Nevertheless, in the planar case, the results of \cite{benediktdrabekgirg} and \cite{bobkovdrabek} allow us to characterize Theorem~\ref{thm:main1} in a more precise way. 
For visual simplicity we denote
$$ 
\lambda_\ominus(p):= \tau_1(p;B^2), 
\quad 
\lambda_\oplus(p):= \tau_2(p;B^2), 
\quad 
\lambda_\circledcirc(p) := \mu_2(p;B^2).
$$
Recall that if $p=2$, then
\begin{equation*}
\lambda_2(2;B^2)=\lambda_3(2;B^2)=\lambda_\ominus(p)
~<~
\lambda_4(2;B^2)=\lambda_5(2;B^2)=\lambda_\oplus(p)
~<~
\lambda_6(2;B^2)=\lambda_\circledcirc(p).
\end{equation*}
For $p>1$ we have the following result.
\medskip

\begin{prop}\label{prof:detalization_of_main_theorem_in_2D_case}
	Let $N=2$. Then for every $p>1$ it holds
	\begin{equation}\label{eq:detalization_of_l2k+1tk_in_2D_case}
	\lambda_2(p;B^2) \leq \lambda_3(p;B^2) \leq \lambda_\ominus(p) < \lambda_\circledcirc(p),
	\end{equation}
	that is, any third eigenfunction on the disc is not radially symmetric. 
	Moreover, there exists $p_1>1$ such that
	\begin{equation}\label{eq:l4<l5<lrad}
	\lambda_4(p;B^2) \leq \lambda_5(p;B^2) \leq \lambda_\oplus(p;2) < \lambda_\circledcirc(p;2)
	\quad \text{ for all } p > p_1,
	\end{equation}
	that is, fourth and fifth eigenfunctions on the disc are also not radially symmetric for $p>p_1$. 
\end{prop}

Note that the last inequality in \eqref{eq:l4<l5<lrad} is reversed for $p$ close to $1$, see~\cite[Theorem~1.3]{bobkovdrabek}.

\medskip

Consider now a bounded domain $\Omega \subset \mathbb{R}^N$ which is invariant under rotation of $N-l$ variables for some $l \in \{1,\dots,N-1\}$, see the definition \eqref{eq:domain_of_revolution} below. Analogously to the case of $N$-ball, it is possible to define symmetric eigenvalues $\tau_k(p;\Omega)$ of the $p$-Laplacian on $\Omega$ for any $k \in \mathbb{N}$, see Section~\ref{sec:Eigenvalue_auxiliary_facts} below. Similarly to Theorem~\ref{thm:main1}, we have the following facts.
\medskip
\begin{prop}\label{prop:main2}
	Let $\Omega \subset \R^N$ be a bounded domain of $N-l$ revolutions defined by \eqref{eq:domain_of_revolution}, where $N \geq 2$ and $l \in \{1,\dots,N-1\}$. Let $p>1$ and $k \geq 1$.
	Then the following inequalities are satisfied:
	\begin{align}
	\label{eq:ln+1<t2_excentric}
	\lambda_2(p;\Omega) \leq \dots \leq \lambda_{N-l+2}(p;\Omega) &\leq \tau_1(p;\Omega);\\
	\label{eq:l2k+1tk1_excentric}
	\lambda_{2k}(p;\Omega) \leq \lambda_{2k + 1}(p;\Omega) &\leq \tau_k(p;\Omega).
	\end{align}	
\end{prop}
\medskip

The article is organized as follows. 
In Section~\ref{sec:HomAlg}, we recall some facts from Algebraic Topology and prove necessary technical statements.
Section~\ref{sec:Eigenvalue_auxiliary_facts} is mainly devoted to the construction of symmetric eigenvalues on domains of revolution. 
Section~\ref{sec:proofs} contains the proofs of the main results.
Finally, in Section~\ref{sec:open_problems}, we discuss the limit cases $p=1$ and $p=\infty$ and some naturally appeared open problems.

\section{Preliminaries}\label{sec:preliminaries}

\subsection{Some algebraic topological results}\label{sec:HomAlg}

Recall first that a subset $X$ of a topological vector space is \textit{symmetric} if it is invariant under the central symmetry map $\iota$ defined as $\iota(x)=-x$. A map $f$ between symmetric sets is called \emph{odd} if $f \circ \iota = \iota \circ f$, and it will be called \emph{even} if $f\circ\iota=f$. 
In the following, we assume all maps to be continuous.

Let us denote by $H_k(X)$ the $k^\textrm{th}$ homology group (over $\Z$) of a manifold $X$ (cf.~\cite[Chapter~2]{hatcher}).
We say that a manifold is an \emph{$n$-manifold} (with $n \in \mathbb{N}$) if it is an oriented closed $n$-dimensional manifold.
If $X$ is an $n$-manifold, then it can be shown that $H_n(X) \cong \Z$ \cite[Theorem~3.26]{hatcher} with a preferred generator given by the orientation of $X$. Moreover, by post-composition, any map $f:X\to Y$ induces linear maps $f_k:H_k(X)\to H_k(Y)$ for each $k\in\N$. When both $X$ and $Y$ are $n$-manifolds, the \emph{degree of the map $f$} is defined as the image by $f_n$ of the preferred generator of $H_n(X)$ in $H_n(Y)\cong\Z$ and denoted as $\deg(f)$. It follows directly from the definitions that if $f:X\to Y$ and $g:Y \to Z$ are two continuous maps between $n$-manifolds, then $\deg(g \circ f) = \deg(g) \deg(f)$. Moreover, two \emph{homotopic} maps, that is two maps with a continuous path of maps between them, have the same degree since they induce the same map on homology; see \cite[Theorem 2.10]{hatcher} and point (c) in \cite[p.134]{hatcher}.

The following result is known as \textit{Borsuk's Theorem} and it was proved in \cite[Hilfssatz~6]{borsuk}. 
An English written proof can found in \cite[Proposition~2B.6]{hatcher}.
\begin{thm}\label{thm:Borsuk}
  Any odd map $f:S^n\to S^n$ has an odd degree. 
\end{thm}
\begin{rem}\label{rem:Borsuk_Ulam_classic}
  Borsuk's Theorem  implies the classical Borsuk-Ulam Theorem which states that there
  is no odd map from a sphere into a sphere of strictly lower dimension. 
\end{rem}

The following proposition is considered as well-known in the literature, see, e.g., \cite[Exercice 14, p. 156]{hatcher}.
\begin{prop}\label{prop:EvenBorsuk}
    Any even map $f:S^n\to S^n$ has an even degree. 
\end{prop}

The following lemma, which will be crucial for our arguments, is a consequence of Borsuk's Theorem.

\begin{lem}\label{lem:ObstructionLemma}
  Let $X$ be a symmetric subset of a topological space. Suppose that there is a map $f:S^n\times [0,1]\to X$
  such that $f_{|S^n\times\{0\}}$ is odd, and either of the following conditions is satisfied:
  \begin{enumerate}
   \item[(a)] $f_{|S^n\times\{1\}}$ is even;
   \item[(b)] $f_{|S^n\times\{1\}}$ is equal to $f_{|S^n\times\{0\}}\circ g$, where $g:S^n\to S^n$ is a map such that $\deg(g)\neq 1$.
  \end{enumerate}
 Then there is no odd map from $X$ to $S^k$ for $k\leq n$.
\end{lem}
\begin{proof}
  Assume, by contradiction, that there exists an odd map $h:X\to S^k$ for some $k\leq n$. 
  By considering $S^k$ as an iterated equator of
  $S^n$, $f$ can be promoted as an odd map $h:X\to S^n$.
  Since $\big(t\mapsto h\circ f_{|S^n\times\{t\}}\big)$ is a continuous map from $h\circ f_{|S^n\times\{0\}}$ to $h\circ f_{|S^n\times\{1\}}$, it follows that they are homotopic and hence have the same degree $d$. 
  Moreover, since $h\circ f_{|S^n\times\{0\}}:S^n\to S^n$ is an odd map, it follows from
  Theorem \ref{thm:Borsuk} that $d$ is odd. Now we distinguish the two cases:
  \begin{enumerate}
   \item[(i)] Under assumption $(a)$, if $f_{|S^n\times\{1\}}$ is even, then so
 is $h\circ f_{|S^n\times\{1\}}:S^n\to S^n$ and hence $d$ is even by Proposition \ref{prop:EvenBorsuk}.
   \item[(ii)] Under assumption $(b)$, we use the multiplicativity of the degree to get 
$$
d=\deg(h\circ f_{|S^n\times\{1\}})=\deg(h\circ f_{|S^n\times\{0\}}\circ g)=\deg(h\circ
  f_{|S^n\times\{0\}})\deg(g)=d \cdot \deg(g) \neq d,
$$ 
since $\deg(g)\neq 1$ by assumption, and $d\neq 0$ since it is odd.
  \end{enumerate}
In both cases, we get a contradiction, and hence the lemma follows.
\end{proof}
\begin{rem}
  It is possible to obtain a weaker result by using the classical Borsuk-Ulam Theorem, without any assumptions on $f_{|S^n\times\{1\}}$. In this case, one can only prove nonexistence of odd maps from $X$ to $S^k$ for $k \leq n-1$.
\end{rem}

To be applied, Lemma \ref{lem:ObstructionLemma} requires an evaluation of the degree of the map $g$. We address now a very elementary example that will be useful to prove Proposition \ref{prop:first_inequality} below.
For that purpose, we consider the permutation map $\tau:S^n\to S^n$ defined by 
  $\tau(x_1,x_2,\ldots,x_{n+1})=(x_{n+1},x_1,\ldots,x_n)$.
\begin{lem}\label{lem:Degrees}
  The map $\tau$ has degree $(-1)^n$.
\end{lem}
\begin{proof}
 As auxiliary maps, we define $\rho_1$ the reflexion along the first coordinate, and $\theta_i$ the rotation of angle $\frac\pi2$ in the oriented plane generated by the $i^\textrm{th}$ and the $(i+1)^\textrm{th}$ coordinates. More explicitly, we have $\rho_1(x_1,x_2,\ldots,x_{n+1}) = (-x_1,x_2,\ldots,x_{n+1})$ and
 $$
 \theta_i(x_1,\ldots,x_{i-1},x_i,x_{i+1},x_{i+2}\ldots,x_{n+1})=(x_1,\ldots,x_{i-1},-x_{i+1},x_i,x_{i+2},\ldots,x_{n+1}).
 $$
 It is then directly computed that
$$
\tau=\left\{
  \begin{array}{ll}
    \theta_1\circ\cdots\circ\theta_{n}&\textrm{for $n$ even},\\
    \rho_1\circ\theta_1\circ\cdots\circ\theta_{n}&\textrm{for $n$ odd}.
  \end{array}\right.
$$
It is easily seen that $\deg(\rho_1)=-1$, cf.\ \cite[Section 2.2, Property (e), p.~134]{hatcher}. Moreover, all rotations are path-connected to the identity map and hence they have degree $1$ by the same codomain $\Z$ argument as in the proof of Lemma~\ref{lem:ObstructionLemma}. Combined with the multiplicativity of the degree, this proves the statement. 
\end{proof}

\subsection{The eigenvalue problem}\label{sec:Eigenvalue_auxiliary_facts}

First we give the following well-known fact. 
\begin{lem}\label{lem:Krasnoselskii_by_scaling}
	Let $w \in W_0^{1,p}(\Omega)$ be such that $w = w_1 +\dots+ w_k$, where $w_i$ and $w_j$ have disjoint supports for $i \neq j$ and each $w_i \in \mathcal{S}_p$. 
	Then 
	$$
	\mathcal{C}_k := 
	\biggl\{
	\sum_{i=1}^{k} \alpha_i w_i\,\big|\, \sum_{i=1}^{k} |\alpha_i|^p = 1
	\biggr\} \subset \mathcal{S}_p,
	$$
	$\mathcal{C}_k$ is symmetric and compact, and $\gamma(\mathcal{C}_k) = k$. 
	Moreover,
	$$
	\max_{u \in \mathcal{C}_k} \int_{\Omega} |\nabla u|^p \, dx 
	\leq 
	\max\Bigl\{\int_{\Omega} |\nabla w_1|^p \, dx, \dots, \int_{\Omega} |\nabla w_k|^p \, dx\Bigr\}.
	$$
	In particular, if $w$ is an eigenfunction of the $p$-Laplacian on $\Omega$ associated to an eigenvalue $\lambda$, and $w$ has at least $k$ nodal domains, then 
	$$
	\lambda_k(p;\Omega) \leq \max\limits_{u \in \mathcal{C}_k} \int_{\Omega} |\nabla u|^p \, dx  \leq \lambda.
	$$
\end{lem}
\begin{proof}
	Since all the statements are trivial, we will prove, for the sake of completeness, only that $\gamma(\mathcal{C}_k) = k$; see~\cite[Proposition~7.7]{rabinowitz}. Note first that there exists an odd homeomorphism $f$ between $\mathcal{C}_k$ and $S^{k-1}$ given by 
	$$
	f\biggl(\sum_{i=1}^{k} \alpha_i w_i\biggr) = \left(|\alpha_1|^{\frac{p}{2}-1}\alpha_1,\dots,|\alpha_k|^{\frac{p}{2}-1}\alpha_k\right).
	$$ 
	This implies that $\gamma(\mathcal{C}_k) \leq k$. If we suppose that $\gamma(\mathcal{C}_k) = n < k$, then there exists a continuous odd map $g:\mathcal{C}_k \to S^{n-1}$. However, the composition $g \circ f^{-1}$ is odd and maps $S^{k-1}$ into $S^{n-1}$ which contradicts the classical Borsuk-Ulam Theorem, cf.\ Remark~\ref{rem:Borsuk_Ulam_classic}.
	Thus, $\gamma(\mathcal{C}_k) = k$.
\end{proof}

Now we generalize the construction of eigenvalues $\tau_k(p;B^N)$ and corresponding symmetric eigenfunctions given in \cite{anoopdrabeksasi} to domains of revolution. 
Let us introduce the usual spherical coordinates in $\mathbb{R}^N$:
\begin{align*}
x_1 &= r \cos \theta_1,\\
x_2 &= r \sin \theta_1 \cos \theta_2,\\
&\cdots \\
x_{N-1} &= r \sin \theta_1 \sin \theta_2 \dots \sin \theta_{N-2} \cos \theta_{N-1},\\
x_{N} &= r \sin \theta_1 \sin \theta_2 \dots \sin \theta_{N-2} \sin \theta_{N-1},
\end{align*}
where $r \in [0,+\infty)$, $(\theta_1, \dots, \theta_{N-2}) \in [0, \pi]^{N-2}$ and $\theta_{N-1} \in [0, 2\pi)$. 
We say that $\Omega \subset \mathbb{R}^N$, $N \geq 2$, is a bounded \textit{domain of $N-l$ revolutions}, if $\Omega$ is a bounded domain and there exists a set $\mathcal{O} \subset [0,+\infty) \times [0, \pi]^{l-1}$ with $l \in \{1, \dots, N-1\}$ such that
\begin{equation}\label{eq:domain_of_revolution}
\Omega = 
\Bigl\{
x \in \mathbb{R}^N \,\big|\, (r,\theta_1,\dots,\theta_{l-1}) \in \mathcal{O},~ (\theta_{l},\dots,\theta_{N-2}) \in [0,\pi]^{N-l-1}, ~ \theta_{N-1} \in [0, 2\pi)
\Bigr\}. 
\end{equation}
Note that the latter two constraints describe a unit sphere $S^{N-l}$. 
Moreover, if $l=1$, then $\Omega$ is radially symmetric.

For any $k \in \mathbb{N}$ consider $2k$ wedges of $\Omega$ defined as (cf.\ Figure~\ref{Fig}) 
\begin{equation}\label{eq:spherical_wedge}
\mathcal{W}_i(k) := 
\Bigl\{
x \in \Omega \,\big|\, \frac{(i-1) \pi}{k} < \theta_{N-1} < \frac{i \pi}{k}
\Bigr\},
\quad 
i \in \{1, \dots, 2k\}.
\end{equation}

\begin{figure}[!h]
	\centering
	\includegraphics[width=0.4\linewidth]{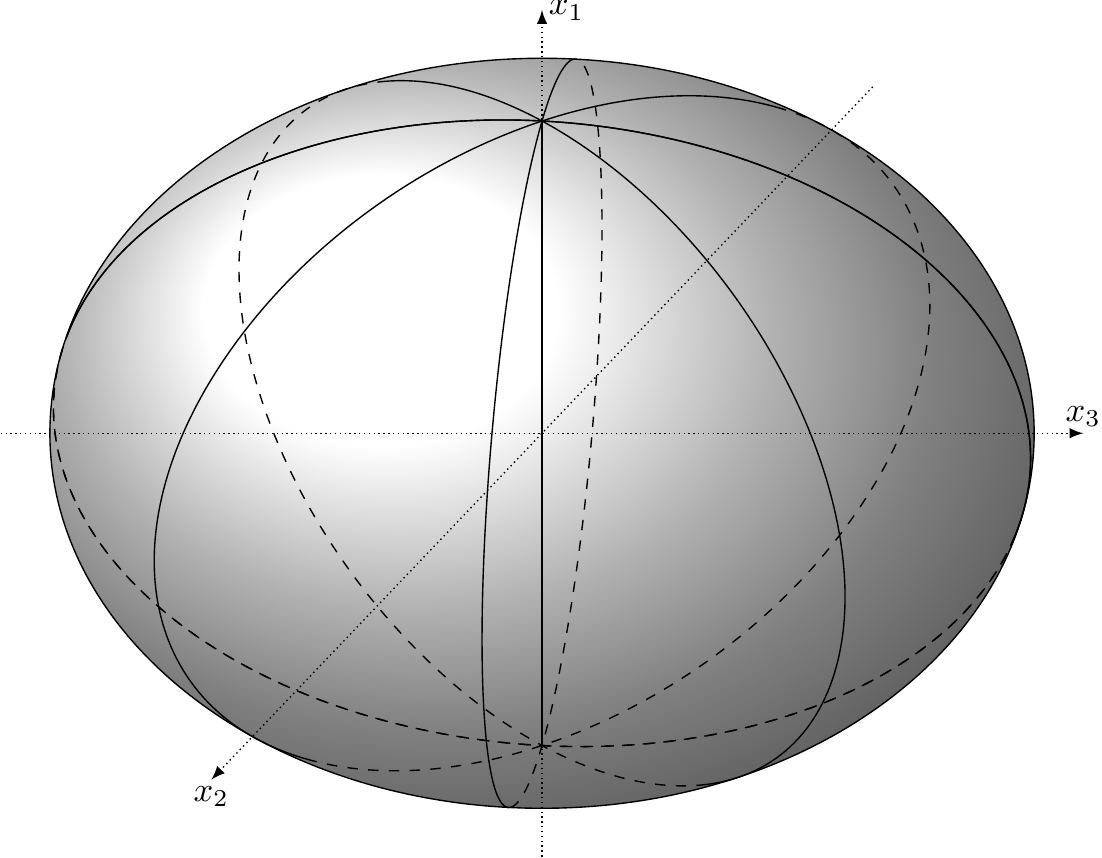}
	\caption{Partitioning of an ellipsoid $\Omega \subset \mathbb{R}^3$ on eight wedges $\mathcal{W}_1(8), \dots, \mathcal{W}_8(8)$.
	(The drawing is based on \cite{trzeciak}.)}
	\label{Fig}
\end{figure}

Let $v \in W_0^{1,p}(\mathcal{W}_1(k))$ be a first eigenfunction of the $p$-Laplacian on $\mathcal{W}_1(k)$ and $\lambda_1(p;\mathcal{W}_1(k))$ be the associated first eigenvalue. 
Hereinafter, we assume that $v$ is extended by zero outside of its support. 
We define
\begin{equation}\label{definitiontau} 
\tau_k(p;\Omega) := \lambda_1(p;\mathcal{W}_1(k)).
\end{equation}
Let $R_\omega(x)$ be the rotation of $x \in \mathbb{R}^N$ on the angle of measure $\omega \in \R$ with respect to $\theta_{N-1}$, that is,
$$
R_\omega(x) = (x_1,\dots,x_{N-2}, r \sin \theta_1 \dots \sin \theta_{N-2} \cos (\theta_{N-1}+\omega), r \sin \theta_1 \dots \sin \theta_{N-2} \sin (\theta_{N-1}+\omega)).
$$
Denote by $v_\omega \in W_0^{1,p}(R_\omega(\mathcal{W}_1(k)))$ the corresponding rotation of $v$, that is,
\begin{equation}\label{eq:rotation}
v_\omega(x) = v(R_{-\omega}(x))
\quad \text{for all } x \in R_\omega(\mathcal{W}_1(k)).
\end{equation}
Consider the function $\Psi_k \in W_0^{1,p}(\Omega)$ given by
\begin{equation}\label{eq:Psi_k}
\Psi_k = v - v_{\frac{\pi}{k}} + v_{\frac{2\pi}{k}} - \dots - v_{\frac{(2k-1)\pi}{k}} \equiv \sum\limits_{i=1}^{2k} (-1)^{i-1} v_{\frac{(i-1) \pi}{k}}.
\end{equation}
\begin{lem}\label{lem:symmetric_eigenvalue}
	$\Psi_k$ is an eigenfunction of the $p$-Laplacian on $\Omega$ associated to the eigenvalue $\tau_k(p;\Omega)$.
\end{lem}
\begin{proof}
	Note that $R_{\frac{i\pi}{k}}(\mathcal{W}_j(k))=\mathcal{W}_{m}(k)$, where $i \in \mathbb{N}$, $j, m \in \{1,\dots,2k\}$ and $m \equiv j+i \ \ (\text{mod } 2k)$.
	Moreover, if we denote by $\sigma_{H_i}(\mathcal{W}_j(k))$ the reflection of $\mathcal{W}_j(k)$ with respect to the hyperplane $H_i := \{x \in \mathbb{R}^N\,\big|\, \theta_{N-1} = \frac{i\pi}{k}\}$, then it is not hard to see that $\sigma_{H_i}(\mathcal{W}_j(k)) = \mathcal{W}_{s}(k)$, where $i \in \mathbb{N}$, $j, s \in \{1,\dots,2k\}$ and $s \equiv 2i-j+1 \ \ (\text{mod } 2k)$.
	At the same time, since the $p$-Laplacian is invariant under orthogonal changes of variables, we obtain that the rotation $v_{\frac{\pi}{k}}$ of $v$ is a first eigenfunction of the $p$-Laplacian on $\mathcal{W}_2(k)$. Analogously, if $w$ is a reflection of $v$ with respect to the hyperplane $H_1$, then $w$ is also a first eigenfunction on $\mathcal{W}_2(k)$. Since the first eigenvalue is simple, we conclude that $w \equiv v_{\frac{\pi}{k}}$. Now, the proof of \cite[Theorem~1.2]{anoopdrabeksasi} based on reflection arguments can be applied with no changes to conclude the desired fact.	
\end{proof}

\begin{rem}
	Let $(\Psi_k)_\omega$ be obtained by rotating $\Psi_k$ on the angle of measure $\omega \in \R$ with respect to $\theta_{N-1}$, see~\eqref{eq:rotation}. Since the $p$-Laplacian and $\Omega$ are invariant under such rotation, we see that $(\Psi_k)_\omega$ is also an eigenfunction associated to $\tau_k(p;\Omega)$. 
\end{rem}

\section{Proofs of the main results}\label{sec:proofs}

The proofs of Theorem~\ref{thm:main1} and Propositions \ref{prof:detalization_of_main_theorem_in_2D_case} and \ref{prop:main2} will be achieved in several steps. 
First, in Proposition~\ref{prop:first_inequality}, we prove the inequalities \eqref{eq:l2k+1tk1_excentric} of Proposition~\ref{prop:main2}. 
The inequalities \eqref{eq:l2k+1tk1} of Theorem~\ref{thm:main1}, being a partial case of \eqref{eq:l2k+1tk1_excentric}, will be hence covered. 
Second, in Proposition~\ref{prop:eq:ln+1<t2}, we prove the inequalities \eqref{eq:ln+1<t2} of Theorem~\ref{thm:main1}. The method of proof carries over to the inequalities \eqref{eq:ln+1<t2_excentric} of Proposition~\ref{prop:main2}, see Proposition~\ref{prop:ln+1<t2_excentric}. 
Finally, we give the proof of Proposition \ref{prof:detalization_of_main_theorem_in_2D_case}. 

\begin{prop}\label{prop:first_inequality}
	Let $\Omega \subset \mathbb{R}^N$ be a bounded domain of $N-l$ revolutions defined by \eqref{eq:domain_of_revolution}, where $N \geq 2$ and $l \in \{1,\dots,N-1\}$. 
	For any $p>1$ and $k \in \mathbb{N}$ it holds
	\begin{equation}\label{eq:l2k+1tk}
	 \lambda_{2k+1}(p;\Omega) \leq \tau_k(p;\Omega).
	\end{equation}
\end{prop}
\begin{proof}
	Denote by $v$ a first eigenfunction of the $p$-Laplacian on the wedge $\mathcal{W}_1(k)$ defined by~\eqref{eq:spherical_wedge} and assume that $v$ is normalized such that $\|v\|_{L^p(\mathcal{W}_1(k))} = 1$. 
	Then $v$ generates the eigenfunction $\Psi_k$ of the $p$-Laplacian on $\Omega$, as defined by~\eqref{eq:Psi_k}, associated to the eigenvalue $\tau_k(p;\Omega)$, see Lemma~\ref{lem:symmetric_eigenvalue}. Note that $\Psi_k$ has exactly $2k$ nodal domains. 
	Consider the set
	$$
	\mathcal{A} := 
	\biggl\{
	\sum_{i=1}^{2k} \alpha_i\,v_{\gamma+\frac{(i-1)\pi}{k}}\,\big|\, \sum_{i=1}^{2k} |\alpha_i|^p = 1,\, \gamma \in \R 
	\biggr\},
	$$
	where $v_\varphi$ is obtained by rotating $v$ on the angle of measure $\varphi \in \R$ with respect to $\theta_{N-1}$, see~\eqref{eq:rotation}. 
	It is not hard to see that $\mathcal{A}$ is symmetric, compact and $\mathcal{A} \subset \mathcal{S}_p$. 
	Consider the continuous map $f:S^{2k-1} \times [0,1] \to \mathcal{A}$ defined by 
	$$
	f\left(\left(|\alpha_1|^{\frac{p}{2}-1}\alpha_1,\ldots,|\alpha_{2k}|^{\frac{p}{2}-1}\alpha_{2k}\right),t\right)= \sum_{i=1}^{2k} \alpha_i \, v_{\frac{t \pi}{k}+ \frac{(i-1)\pi}{k}},
	\quad \text{where } \sum_{i=1}^{2k} |\alpha_i|^p = 1.
	$$ 
	Then, $f$ clearly satisfies  $f_{|S^{2k-1}\times\{0\}}\circ\iota=\iota\circ f_{|S^{2k-1}\times\{0\}}$ and, in view of~\eqref{eq:Psi_k}, $f_{|S^{2k-1}\times\{1\}}=f_{|S^{2k-1}\times\{0\}}\circ \tau$, where $\iota$ and $\tau$ are defined in Section \ref{sec:HomAlg}. Therefore, it follows from assertion $(b)$ of Lemma \ref{lem:ObstructionLemma} and Lemma \ref{lem:Degrees} that there is no odd map from $\mathcal{A}$ to $S^{n}$ for any $n \leq 2k-1$, which implies that $\gamma(\mathcal{A}) \geq 2k+1$.
	Thus, $\mathcal{A} \in \Gamma_{2k+1}(p)$. 

	Noting now that for any $u \in \mathcal{A}$ it holds
	$$
	\int_{\Omega} |\nabla u|^p \, dx = \sum_{i=1}^{2k}
	|\alpha_i|^p \int_{\Omega} \left|\nabla v_{\gamma+\frac{(i-1)\pi}{k}}\right|^p  dx = 
	\sum_{i=1}^{2k}
	|\alpha_i|^p \, \tau_k(p;\Omega) = \tau_k(p;\Omega),
	$$
	we conclude the desired inequality:
	$$ 
	\lambda_{2k+1}(p;\Omega) \leq \max_{u \in \mathcal{A}} \int_{\Omega} |\nabla u|^p \, dx = 
	\tau_k(p;\Omega).
	$$
\end{proof}

\begin{rem}\label{rem:2k+1_linear}
	In the linear case $p=2$, the inequality \eqref{eq:l2k+1tk} can be easily obtained using the Courant-Fisher variational principle \eqref{eq:eigenvalue_linear}. 
	Indeed, since the Laplacian is rotation invariant and $\Omega$ is a domain of revolution, for any $i \geq 1$ we can find at least two linearly independent symmetric eigenfunctions associated to $\tau_i(2;\Omega)$, one is a rotation of another. Therefore, taking a first eigenfunction and also two linearly independent eigenfunctions for every $i \in \{1,\dots,k\}$, we produce a $(2k+1)$-dimensional subspace of $W_0^{1,2}(\Omega)$ which leads to the desired inequality via \eqref{eq:eigenvalue_linear}. 
	Let us also remark that, in view of Pleijel's Theorem, the inequality \eqref{eq:l2k+1tk} is strict for sufficiently large $k \in \mathbb{N}$, see, e.g., \cite{helffer}.
\end{rem}

\begin{rem}\label{remark:connection_between_multiplicity_and_nodal_set}
	Let, for simplicity, $N=2$, $\Omega = B^2$ and $k=1$. Assume that there exists a second eigenfunction $\phi$ of the $p$-Laplacian on $\Omega$ which is antisymmetric with respect to the rotation of the angle $\pi$, that is, $\phi_{\pi} = -\phi$. (This happens, for instance, when the nodal set is a diameter or a ``yin-yang''-type curve.) Then the proof of Proposition~\ref{prop:first_inequality} works with no changes considering $\phi^+$ or $\phi^-$ instead of $v$, which yields $\lambda_2(p;B^2) = \lambda_3(p;B^2)$. Therefore, the knowledge about structure of the nodal set of higher eigenfunctions plays an important role for our arguments.
\end{rem}

It is of independent interest to prove the inequalities~\eqref{eq:ln+1<t2} of Theorem~\ref{thm:main1} up to $\lambda_N(p;\Omega)$, since the proof uses only rotations of $\Psi_1$ to increase the Krasnosel'ski\u{\i} genus.
\begin{prop}\label{prop:ln<t2}
	Let $\Omega \subset \mathbb{R}^N$ be a bounded radially symmetric domain, $N \geq 2$. 
	Then for any $p>1$ it holds 
	\begin{equation*}\label{eq:lNleqt1}
	\lambda_{N}(p;\Omega) \leq \tau_1(p;\Omega).
	\end{equation*}	
\end{prop}
\begin{proof}
	For any $x \in S^{N-1}$ we define 
	$$
	\Omega_x := \{z \in \Omega\,\big|\, \langle z, x \rangle > 0\}.
	$$
	Denote as $v_x$ the first eigenfunction on $\Omega_x$ such that $v_x > 0$ in $\Omega_x$ and $\|v_x\|_{L^p(\Omega_x)} = 1$, and extend it by zero outside of $\Omega_x$.
	Arguing as in Lemma~\ref{lem:symmetric_eigenvalue}, it can be deduced that $\frac{v_x - v_{-x}}{\sqrt[p]{2}}$ is an eigenfunction associated to $\tau_1(p;\Omega)$ for any $x \in S^{N-1}$.
	Consider the set
	$$
	\mathcal{A} := \Bigl\{\frac{v_x - v_{-x}}{\sqrt[p]{2}}\,\big|\, x \in S^{N-1} \Bigr\}.
	$$
	It is not hard to see that $\mathcal{A}$ is compact. Moreover, $\mathcal{A}$ is evidently symmetric and $\mathcal{A} \subset \mathcal{S}_p$.
	Note that $x$ is uniquely determined by the choice of $\frac{v_x-v_{-x}}{\sqrt[p]{2}}$ since $x$ corresponds to the unique unit normal vector of the nodal set which points to the nodal domain $\Omega_x$. 
	Therefore, taking $h:\mathcal{A} \to S^{N-1}$ defined by $h\left(\frac{v_x - v_{-x}}{\sqrt[p]{2}}\right) = x$, we deduce that $h$ is an odd homeomorphism, and hence $\gamma(\mathcal{A}) \leq N$. If we suppose that $\gamma(\mathcal{A}) < N$, then we get a contradiction as in the proof of Lemma~\ref{lem:Krasnoselskii_by_scaling}. Therefore, $\gamma(\mathcal{A}) = N$ and $\mathcal{A} \in \Gamma_N(p)$, and we conclude as in the proof of Proposition~\ref{prop:first_inequality}.
\end{proof}

To prove the whole chain of inequalities \eqref{eq:ln+1<t2} of Theorem~\ref{thm:main1}, we combine rotations of $\Psi_1$ with the scaling of its nodal components. 
\begin{prop}\label{prop:eq:ln+1<t2}
	Let $\Omega \subset \mathbb{R}^N$ be a bounded radially symmetric domain, $N \geq 2$. 
	Then for any $p>1$ it holds 
	\begin{equation*}
	 \lambda_{N+1}(p;\Omega) \leq \tau_1(p;\Omega).
	\end{equation*}
\end{prop}
\begin{proof}
	Using the notation $v_x$ from Proposition~\ref{prop:ln<t2}, we define the set 
	$$
	\mathcal{A} := \left\{ \alpha_1\,v_x+\alpha_2\,v_{-x}\,\big|\, |\alpha_1|^p + |\alpha_2|^p = 1,\,x \in S^{N-1} \right\}.
	$$
	As before, $\mathcal{A} \subset \mathcal{S}_p$ and $A$ is symmetric and compact.
	Let $\gamma: [0,1] \to \{z \in \mathbb{R}^2: |z_1|^p + |z_2|^p = 1\}$ be a path from $\Big(\frac{1}{\sqrt[p]{2}},-\frac{1}{\sqrt[p]{2}}\Big)$ to $\Big(\frac{1}{\sqrt[p]{2}},\frac{1}{\sqrt[p]{2}}\Big)$ and denote by $\gamma_1(t)$ and $\gamma_2(t)$ the first and the second component of $\gamma(t)$, respectively. 
	The continuous map $f:S^{N-1} \times [0,1] \to \mathcal{A}$ defined by $f(x,t)=\gamma_1(t)v_x+\gamma_2(t)v_{-x}$ clearly satisfies $f_{|S^{N-1}\times\{0\}}\circ\iota=\iota\circ f_{|S^{N-1}\times\{0\}}$ and $f_{|S^{N-1}\times\{1\}}\circ\iota=f_{|S^{N-1}\times\{1\}}$, where $\iota$ is defined in Section~\ref{sec:HomAlg}. Then, it follows from assertion $(a)$ of Lemma \ref{lem:ObstructionLemma} that there is no odd map from $\mathcal{A}$ to $S^{n-1}$ for any $n \leq N$, and hence $\gamma(\mathcal{A})\geq N+1$. 
	Thus $\mathcal{A} \in \Gamma_{N+1}(p)$, and we conclude as in the proof of Proposition~\ref{prop:first_inequality}.
\end{proof}

\begin{cor}
	If $\lambda_2(p;\Omega)=\tau_1(p;\Omega)$, then the second eigenvalue has multiplicity at least $N$.
\end{cor}

The inequalities \eqref{eq:ln+1<t2_excentric} of Proposition~\ref{prop:main2} can be proved in much the same way as Proposition~\ref{prop:eq:ln+1<t2}. Let us briefly sketch the proof.  
\begin{prop}\label{prop:ln+1<t2_excentric}
	Let $\Omega \subset \mathbb{R}^N$ be a bounded domain of $N-l$ revolutions, where $N \geq 2$ and $l \in \{1,\dots,N-1\}$.
	Then for any $p>1$ it holds 
	$$
	\lambda_{N-l+2}(p;\Omega) \leq \tau_1(p;\Omega)
	$$
\end{prop}
\begin{proof}
	Take any $x \in S^{N-l}$ and define a hemisphere 
	$$
	S^{N-l}_x := \{y \in S^{N-l}\,\big|\, \langle x,y \rangle > 0\}.
	$$
	We parametrize $S^{N-l}_x$ in spherical coordinates by angles $(\theta_l, \dots, \theta_{N-1})$ and define 
	$$
	\Omega_x := \{z \in \Omega\,\big|\, (\theta_l, \dots, \theta_{N-1}) \in S^{N-l}_x\}.
	$$
	Denote as $v_x$ the first eigenfunction on $\Omega_x$ such that $v_x > 0$ in $\Omega_x$ and $\|v_x\|_{L^p(\Omega_x)} = 1$. 
	In view of the symmetries of $\Omega$ (see \eqref{eq:domain_of_revolution}) it is not hard to obtain that $v_x$ is associated to the eigenvalue $\lambda = \tau_1(p;\Omega)$ for any $x \in S^{N-l}$. Consider the set
	$$
	\mathcal{A} := \{ \alpha_1\,v_x+\alpha_2\,v_{-x}\,\big|\, |\alpha_1|^p + |\alpha_2|^p =
	1,\,x \in S^{N-l} \}.
	$$
	The rest of the proof goes along the same lines as in Proposition~\ref{prop:eq:ln+1<t2}.
\end{proof}

\bigskip
\noindent
\textbf{Proof of Proposition~\ref{prof:detalization_of_main_theorem_in_2D_case}.}
1) In view of \eqref{eq:ln+1<t2} with $N=2$, to justify \eqref{eq:detalization_of_l2k+1tk_in_2D_case} it is sufficient to show that
	$$
	\lambda_\ominus(p) < \lambda_\circledcirc(p) \quad \text{for any } p>1.
	$$
	This fact was fully proved in \cite{benediktdrabekgirg}, although the case $p \in (1,1.01)$ is not explicitly stated in the text. 
	For the sake of completeness, we collect the arguments from \cite{benediktdrabekgirg} to explain the proof. 
	
	Denote by $B^+$ a half-disc of a unit disc $B^2$. By definition we have $\lambda_\ominus(p) = \lambda_1(p;B^+)$. Translation invariance of the $p$-Laplacian and the strict domain monotonicity of its first eigenvalue (cf.\ \cite[Proposition~4]{benediktdrabekgirg}) imply that $\lambda_\ominus(p) < \lambda_1(p;B^2_{1/2})$, where $B^2_{1/2}$ is a disc of radius $1/2$.
	On the other hand, it is known that $\lambda_\circledcirc(p) = \lambda_1\bigl(p;B^2_{\nu_1(p)/\nu_2(p)}\bigr)$, where $B^2_{\nu_1(p)/\nu_2(p)}$ is a disc of radius $\nu_1(p)/\nu_2(p)$, and $\nu_1(p)$, $\nu_2(p)$ are the first two positive roots of a (unique) solution of the Cauchy problem
	\begin{equation}\label{eq:radial}
	\left\{
	\begin{aligned}
	&-(r|u'|^{p-2}u')' = r |u|^{p-2}u \quad \text{in } (0, +\infty),\\
	&u(0) = 1 \quad u'(0)=0,
	\end{aligned}
	\right.
	\end{equation}
	see \cite[Lemmas~5.2 and 5.3]{delpinomanasevich}. 
	Therefore, if the inequality
	\begin{equation}\label{eq:2nu1<nu2}
	2\nu_1(p)<\nu_2(p)
	\end{equation}
	holds for all $p>1$, then the strict domain monotonicity yields the desired conclusion:
	$$
	\lambda_\ominus(p) < \lambda_1(p;B^2_{1/2}) < 
	\lambda_1\bigl(p;B^2_{\nu_1(p)/\nu_2(p)}\bigr)
	= \lambda_\circledcirc(p).
	$$
	The inequality \eqref{eq:2nu1<nu2} is, in fact, the main objective of \cite{benediktdrabekgirg}. 
	In the interval $p \in [1.01, 226]$, \eqref{eq:2nu1<nu2} was proved in \cite[Proposition~7]{benediktdrabekgirg} via a self-validated numerical integration of \eqref{eq:radial}. For $p > 226$, \eqref{eq:2nu1<nu2} was proved in \cite[Proposition~13]{benediktdrabekgirg} by obtaining analytical bounds for $\nu_1(p)$ and $\nu_2(p)$. In the rest case $p \in (1, 1.01)$ it was shown that $\lambda_\ominus(p) \leq 3.5$, see the proof of \cite[Proposition~6]{benediktdrabekgirg}. This fact was enough to apply the proof of \cite[Theorem~6.1]{parini} and get nonradiality of the second eigenfunction. 
	However, as a byproduct of the proof of \cite[Theorem~6.1]{parini}, we know also that $\lambda_\circledcirc(p) > 3.5$ for $p \in (1, 1.1)$, which yields $\lambda_\ominus(p) < \lambda_\circledcirc(p)$ for $p \in (1, 1.01)$. 
	Thus, summarizing the above facts, we conclude that $\lambda_\ominus(p) < \lambda_\circledcirc(p)$ for all $p>1$. 
		
	2) The first two inequalities in \eqref{eq:l4<l5<lrad} follow from \eqref{eq:l2k+1tk1} by taking $k=2$. The last inequality in \eqref{eq:l4<l5<lrad} was proved in \cite[Theorem~1.2]{bobkovdrabek}.	
\qed

\section{Final remarks and open questions}\label{sec:open_problems}

The results of this paper can be applied also to the singular case $p=1$, which must be treated separately. In \cite{littig} the authors defined a sequence of variational eigenvalues and proved that they can be approximated by the corresponding eigenvalues of the $p$-Laplacian as $p \to 1$. The second variational eigenvalue of the $1$-Laplacian can be characterized geometrically, as a consequence of \cite[Theorem 2.4]{littig} and \cite[Theorem 5.5]{parini} (see also \cite{bobkovparini}). In particular, if $\Omega = B^2$ is a disc, it holds $\lambda_2(1;B^2)=\lambda_\ominus(1;B^2)$, and therefore
\[ \lambda_2(1;B^2)=\lambda_3(1;B^2)=\lambda_\ominus(1;B^2)\]
by reasoning as in Proposition \ref{prop:first_inequality}. That is, the second eigenvalue of the $1$-Laplacian on a disc has multiplicity (in the sense of \eqref{eq:multiplicity}) at least $2$. 

The limit case $p = \infty$ can be also considered in terms of a geometric characterization of the corresponding first and second eigenvalues. 
It is known from \cite{julindman1999} and \cite{julind2005} that 
\[
\lim\limits_{p \to \infty} \lambda_1(p;\Omega)^{\frac{1}{p}} = \frac{1}{R_1}
\quad \text{and} \quad 
\lim\limits_{p \to \infty} \lambda_2(p;\Omega)^{\frac{1}{p}} = \frac{1}{R_2},
\]
where $R_1$ is the radius of a maximal ball inscribed in $\Omega$, and $R_2$ is the maximal radius of two equiradial disjoint balls inscribed in $\Omega$. 
Let $B^N$ be a ball of radius $R$. Then we deduce from \eqref{eq:ln+1<t2} that
$$
\lim_{p \to \infty} \lambda_2(p;B^N)^{\frac{1}{p}} = \dots = 
\lim_{p \to \infty} \lambda_{N+1}(p;B^N)^{\frac{1}{p}} = 	
\lim_{p \to \infty} \tau_1(p;B^N)^{\frac{1}{p}}
\equiv 
\lim_{p \to \infty} \lambda_1(p;\mathcal{W}_1(2))^{\frac{1}{p}} = \frac{2}{R}.
$$

\medskip
We are left with several open problems.
\begin{enumerate}
	\item By analogy with the linear case, it would be interesting to show the optimality of \eqref{eq:ln+1<t2}, namely whether the inequality \[\tau_1(p;\Omega)<\lambda_{N+2}(p;\Omega),\] where $\Omega$ is a radially symmetric bounded domain, holds true.
	\item To prove \eqref{eq:l2k+1tk1} we used the scaling of nodal components of symmetric eigenfunctions corresponding to $\tau_k(p;\Omega)$ together with their rotation with respect to the angle $\theta_{N-1}$. However, it is not hard to see that for $N \geq 3$, symmetric eigenfunctions can be also rotated with respect to all the angles $\theta_i$, where $i \in \{1,\dots,N-1\}$ if $\Omega$ is radial, and $i \in \{2,\dots,N-1\}$ if $\Omega$ is a general domain of revolution. This observation leads to the conjecture that for every $k \geq 1$ there exists $j \geq 2$ such that 
	$$
	\lambda_{2k}(p;\Omega) \leq \dots \leq \lambda_{2k+j}(p;\Omega) \leq \tau_k(p;\Omega).
	$$
	The proof might be achieved by showing the nonexistence of maps $S^{n_1} \times S^{n_2} \to S^{m}$, for suitable $n_1$, $n_2$, $m \in \N$, which are odd in the first variable (corresponding to the normalization constraint) and satisfy some additional conditions given by symmetries of eigenfunctions.
	\item In the spirit of the previous question, it is natural to study a generalization of \eqref{eq:l2k+1tk1} where the upper bound is given by eigenvalues whose associated eigenfunctions are invariant under the action of other symmetry groups. 
	\item Is it possible to obtain multiplicity results for domains $\Omega$ which satisfy different symmetry properties, for instance if $\Omega$ is a square? In this case, on the one hand, numerical evidence \cite{yaozhou} supports the conjecture that $\lambda_2(p;\Omega)<\lambda_3(p;\Omega)$ if $p\neq 2$, unlike the linear case where equality trivially holds. 
	On the other hand, if the nodal set of a second eigenfunction $\varphi_{2,p}$ is a middle line or a diagonal of the square, as indicated again in \cite{yaozhou}, then there is another second eigenfunction linearly independent with $\varphi_{2,p}$ obtained by rotating $\varphi_{2,p}$ by an angle of $\frac{\pi}{2}$. 
\end{enumerate}

\bigskip
\noindent
{\bf Acknowledgments.}
The article was started during a visit of E.P. at the University of West Bohemia and was finished during a visit of V.B. at Aix-Marseille University. The authors wish to thank the hosting institutions for the invitation and the kind hospitality. 
V.B. was supported by the project LO1506 of the Czech Ministry of Education, Youth and Sports.

\end{document}